\documentclass{amsart}

\usepackage{etex}
\usepackage{amsmath, amssymb}
\usepackage{array}
\usepackage[frame,cmtip,arrow,matrix,line,graph,curve]{xy}
\usepackage{graphpap, color, paralist, pstricks}
\usepackage[mathscr]{eucal}
\usepackage[pdftex]{graphicx}
\usepackage[pdftex,colorlinks,backref=page,citecolor=blue]{hyperref}
\usepackage{tikz}
\usepackage{pifont}

\newtheorem{thm}{Theorem}[section]
\newtheorem{prop}[thm]{Proposition}

\theoremstyle{definition}

\theoremstyle{plain} \newtheorem{cor}[thm]{Corollary}
\theoremstyle{plain} \newtheorem{conj}[thm]{Conjecture}
\theoremstyle{plain} \newtheorem{lemma}[thm]{Lemma}
\theoremstyle{plain} \newtheorem{obs}[thm]{Observation}

\newcommand{\gO}{\Omega}

\newcommand{\beq}[1]{\begin{equation}\label{#1}}
\newcommand{\enq}[0]{\end{equation}}

\newcommand{\eps}{\epsilon}
\newcommand{\ga}[0]{\alpha }
\newcommand{\gb}[0]{\beta }
\newcommand{\gc}[0]{\gamma }

\newcommand{\gs}[0]{\sigma }

\newcommand{\mn}[0]{\medskip\noindent}
\newcommand{\nin}[0]{\noindent}

\newcommand{\sub}[0]{\subseteq}
\newcommand{\sm}[0]{\setminus}
\newcommand{\ra}[0]{\rightarrow}

\newcommand{\mis}[0]{\mbox{\rm{mis}}}

\newcommand{\supp}[0]{\mbox{\rm{supp}}}

\newcommand{\E}[0]{\mathbb E}

\begin{document}

\title{An isoperimetric inequality for the Hamming cube and some consequences}

\author{Jeff Kahn \and Jinyoung Park}
\thanks{The authors are supported by NSF Grant DMS1501962 and BSF Grant 2014290.}
\thanks{JK was supported by a Simons Fellowship.}
\email{jkahn@math.rutgers.edu, jp1324@math.rutgers.edu}
\address{Department of Mathematics, Rutgers University \\
Hill Center for the Mathematical Sciences \\
110 Frelinghuysen Rd.\\
Piscataway, NJ 08854-8019, USA}

\begin{abstract}

Our basic result, an isoperimetric inequality for Hamming cube 
$Q_n$, can be written:
\[
\int h_A^\beta d\mu \ge 2 \mu(A)(1-\mu(A)).
\]
Here $\mu$ is uniform measure on $V=\{0,1\}^n$ ($=V(Q_n)$); $\beta=\log_2(3/2)$; and,
for $S\sub V$ and $x\in V$,
\[
h_S(x) = \begin{cases} d_{V \setminus S}(x) &\mbox{ if } x \in S, \\ 0 &\mbox{ if } x \notin S \end{cases}
\]
(where $d_T(x)$ is the number of neighbors of $x$ in $T$).

This implies inequalities involving mixtures of edge and vertex boundaries, 
with related stability results, and suggests some more general possibilities.
One application, a stability result for the set of edges connecting two disjoint subsets of $V$
of size roughly $|V|/2$, is a key step in showing that the number of maximal independent 
sets in $Q_n$ is $(1+o(1))2n\exp_2[2^{n-2}]$.
This asymptotic statement, whose proof will appear separately, was the original
motivation for the present work.

\end{abstract}

\maketitle

\section{Introduction}

We write $Q_n$ for the $n$-dimensional Hamming cube and $V$ for $V(Q_n)$. 
For $T \sub V$ let $d_T(x)$ be the number of neighbors of $x$ in $T$ ($x \in V$) 
and define $h_S:V \rightarrow \mathbb N$ by 
\beq{h} h_S(x) = \begin{cases} d_{V \setminus S}(x) &\mbox{ if } x \in S, \\ 0 &\mbox{ if } x \notin S. \end{cases} \enq 
For $f:V \rightarrow \mathbb N$, a probability measure $\nu$ on $V$ and $X\sub V$, we set
\[ \int_Xf d \nu = \sum_{x \in X} f(x)\nu(x).\]
We also use $\int$ for $\int_V$.

Our main result is the following isoperimetric inequality. Throughout this paper we use 
$\beta$ for $\log_2(3/2)~ (\approx .585)$  and $\mu$ for uniform measure on $V$. 
(A few definitions are given in Section \ref{sec:def}.)

\begin{thm} \label{main:iso} For any $A \subseteq V$,
\begin{equation}\label{ineq:main} \int h_A^\beta d\mu \ge 2 \mu(A)(1-\mu(A)).\end{equation}
\end{thm}

\noindent The form of Theorem \ref{main:iso} is inspired by the following inequality of Talagrand \cite{Tal93}.

\begin{thm} \label{thm:T} For any $A \subseteq V$, 
\[ \int \sqrt{h_A} d\mu \ge \sqrt 2 \mu(A)(1-\mu(A)). \]
\end{thm}

Notice that Theorem \ref{main:iso} is tight in two ways:  it holds with equality for 
subcubes of codimensions 1 and 2, and for subcubes of codimension 2 
it does not hold for any smaller value of $\gb$.  
As far as we know the $\sqrt{2}$ in Theorem \ref{thm:T} 
could be replaced by 2 when $\mu(A)=1/2$ (but of course not in general).
The difference between 2 and $\sqrt{2}$ wouldn't have mattered in
\cite{Tal93}, but getting the right constant when $\mu(A)$ is close to 1/2
was crucial
for applications, particularly the one in \cite{misQ} (Theorem~\ref{misqn} below)
that was our original motivation---see 
the "stability" result Theorem \ref{thm:main} that is the present work's contribution to \cite{misQ}.

Before discussing applications we briefly recall a few basic notions.

\subsection{Definitions} \label{sec:def}

As usual $[n]=\{1,\ldots, n\}$, $\mathbb P$ is the set of positive integers and $x=a\pm b$ means $a-b\leq x\leq a+b$.
We use $A, B,C$ and $W$ 
for subsets of $V$ and $E$ for $E(Q_n)$. 
For $x \in V$, $x_i$ is (as usual) the $i$th coordinate of $x$, and $x^i$ is the 
vertex obtained from $x$ by flipping $x_i$.
 For any $A$,
\[A^i=\{ x^i : x \in A\},\]
the \textit{vertex-boundary} of $A$ is
\[ \partial A =\{x \notin A:x \sim y \mbox{ for some } y \in A\},\]
and the \textit{edge-boundary} of $A$ is
\[ \nabla A =\{(x,y):x \in A, y \notin A\}.\]
We also use
\[\nabla(A,B)=\{(x,y):x \in A, y \in B\},\]
\[\nabla_i A=\{(x,x^i): x \in A, x^i \notin A\},\]
\[\nabla_I A = \cup_{i\in I}\nabla A_i \quad (I \sub [n]),\]
and
\[\nabla_i(A,B)=\{(x,x^i):x \in A, x^i \in B\}.\]
 We say $C$ is \textit{a codimension $k$ subcube} if there are $I \subseteq [n]$ of size $k$
and $z \in \{0,1\}^I$ such that
\[ C=\{ x \in V : x_i=z_i \mbox{ for all } i \in I\}.\]

\subsection{First application: separating the cube}\label{Separating}

Isoperimetric inequalities
beginning with Harper \cite{Harper} (and for edge boundaries also 
Lindsey \cite{Lindsey})
give lower bounds in terms of $|A|$ on the sizes of $\partial A$ and $\nabla A$; e.g.
\beq{edge.iso}
|\nabla A|\geq |A|\log_2(2^n/|A|),
\enq 
with equality iff $A$ is a subcube.
We are interested in hybrid versions of these.  
In what follows we assume $(A,B,W)$ is a partition of $V$, with $W$ thought of as small.
The next two conjectures are a simple illustration of what
we have in mind, followed by something general.
\begin{conj}\label{FWconj}
There is a fixed $K$ such that if $\mu(A)=1/2$, then
\[ 
|\nabla(A,B)|+K \sqrt{n}~ |W| \ge 2^{n-1}.\]
\end{conj}
\nin
With
$ \partial (a) = \min\{|\partial A|: |A|=a\}$
and $\nabla (a)$ defined similarly,
our maximal guess in this direction is:
\begin{conj}
If $|A|=a$, then
\[~   
|\nabla(A,B)|/\nabla(a) + |W|/\partial(a)\geq 1.
\]   
\end{conj}
\nin
Results of Margulis \cite{Margulis} and Talagrand \cite{Tal93}
(motivated by \cite{Margulis}) imply
tradeoffs between $|\nabla A|$ and $|\partial A|$, but don't seem to help here.
Theorem \ref{main:iso} implies a weaker version of Conjecture~\ref{FWconj}:
\begin{cor}\label{cubesep}
For $A,B,W$ as in Conjecture \ref{FWconj},
$~ |\nabla(A,B)|+n^\gb |W| \ge 2^{n-1}.$
\end{cor}

\subsection{Second application: stability for ``almost'' isoperimetric subsets}
\label{2ndapp}
 A simple 
(though now suboptimal) "stability" statement for edge boundaries says:
\begin{thm}\label{ellis}
For a fixed k, if $|A| = 2^{n-k}$ and $|\nabla A| < (1+\eps) |A|\log_2(2^n/|A|)$, then
there is a subcube $C$ with $\mu(C\Delta A) =O(\eps)$
(where the implied constant depends on $k$).
\end{thm}
\nin
This was proved for $k=1$ by
Friedgut, Kalai and Naor \cite{FKN}; then for $k=2,3$
by Bollob\'as, Leader and Riordan, who conjectured the general statement
(see \cite{Ellis}); and finally in full by Ellis \cite{Ellis}.
These all based on Fourier analysis;
e.g.\ at the heart of \cite{Ellis} is Talagrand's extension \cite{TalRusso} of \cite{KKL}.
Even stronger, very recent results of Ellis, Keevash and Lifshitz \cite{EKL} are more elementary but rather involved.

Notice that if $A$ is (sufficiently) close to a codimension $k$ subcube then there 
is an $I\sub [n]$ of size $k$ with $\nabla A\approx \nabla_I A$. 
In fact the implication goes both ways; this follows (more or less) from Theorem~\ref{ellis}, but 
is also easy without that machine: 
\begin{prop}\label{cube.prop}
Assume $|A|= (1\pm \eps)2^{n-k}$ and
\[
|\nabla A \setminus \nabla_I A| \le \eps|A|,
\]
where $I$ is a $k$-subset of $[n]$.  
Then there is a (codimension $k$) subcube $C$ with $|A\Delta C| = O(\eps)|A|$
(where the implied constant depends on $k$). 
\end{prop}

The original motivation for Theorem \ref{main:iso} arose in connection with our
efforts to prove the following statement, which had been conjectured in \cite{IK}.
Here $\mis(G)$ is the number of maximal independent sets in the graph $G$.
\begin{thm}\label{misqn}
$\mis(Q_n)\sim 2n\exp_2[2^{n-2}]$.
\end{thm}
\nin
The proof of this is completed in \cite{misQ}.
What it needed from isoperimetry (see \cite{misQ} for the connection)
was a variant of Theorem~\ref{ellis}---really, just of the original result of \cite{FKN}---of 
the following type.
\begin{center}
\emph{If $(A, B, W)$ is a partition of $V$ with $\mu(A),\mu(B)\approx 1/2$ 
(so $W$ is ``small") and $|\nabla(A,B)|\approx 2^{n-1}$, then\\
$\nabla A\approx \nabla_iA$ $~$ for some i.}
\end{center}

\nin
Of course this depends on quantification; e.g.\ it can fail with $\mu(W)$ 
as small as $\Theta(n^{-1/2})$ (let $W$ consist of strings of weight $\lfloor n/2\rfloor$).
Note also that here the full edge boundary of $A$ need \emph{not} be small,
since there is no restriction (beyond $n|W|$) on $|\nabla(A,W)|$.

The following consequence of Theorem \ref{main:iso} is a (limited) statement of
the desired type, the 
case $k=1$ of which
suffices for \cite{misQ}.
(Recall $\beta=\log_2(3/2)$.)

\begin{thm} \label{thm:main} 
For $k\in \{1,2\}$ the following holds.
Suppose $(A, B, W)$ is a partition of $V$ with 
$\mu(A) = (1\pm   \eps)2^{-k}$, 
$\mu(W)\le \eps n^{-\beta}$
and
\beq{nabla} 
|\nabla(A,B)|<(1+\eps)k2^{n-k}.
\enq
Then there is $I\sub [n]$ of size $k$ such that
\beq{conc} 
|\nabla_i A| = (1-O(\eps))2^{n-k} ~~\forall i\in I. 
\enq
Furthermore, there is a codimension $k$ subcube $C$ such that
\beq{conc'}\mu(C\Delta A)=O(\eps).
\enq
\end{thm}
\nin

\begin{conj}\label{T1.8conj}
The statement in 
Theorem~\ref{thm:main} holds for all $k\in \mathbb{P}$, even with 
$n^\beta$ replaced by $2^n/\partial(|A|)$.
\end{conj}
\nin
(The implied constant in \eqref{conc} and \eqref{conc'} would necessarily depend on $k$.)

Note Theorem~\ref{thm:main} implies an isoperimetric statement---similar to those in 
Section~\ref{Separating}---of which it is a stability version; namely:

\begin{cor}
For $k\in \{1,2\}$, the assumptions of Theorem~\ref{thm:main} imply 
$|\nabla(A,B)|> (1-O(\eps))k2^{n-k}.$
\end{cor}
\nin
(And of course similarly for whatever one can establish in the direction of Conjecture~\ref{T1.8conj}.)

Finally, the next observation provides a general approach to proving something like the statement in
Theorem~\ref{thm:main} for other values of $k$.  (Its proof is similar to
the derivation of Theorem~\ref{thm:main} from Theorem~\ref{main:iso}
and is omitted.)

\begin{thm} \label{thm:gen}
Fix $k\in \mathbb{P}$ and suppose there are $f,g:[0,1]\ra\Re^+$ such that (i) 
$g$ is continuous with $g(2^{-k})= k2^{-k}$ and (ii) $f$ is increasing and
strictly concave, with $f(0)=0$, $f(k)=k$ and
\[
\int f(h_A)d\mu \geq g(\mu(A)) ~~\forall A\sub V.
\]
Then the conclusions of Theorem~\ref{thm:main} hold
(with implied constants depending on $f$ and $g$)
for $A,B,W$ as in the theorem, except with the bound on $w$ replaced by $w\leq \eps/f(n)$.
\end{thm}
\nin
(For the cases covered by Theorem~\ref{thm:main}, 
Theorem~\ref{main:iso} gives the hypothesis of Theorem~\ref{thm:gen}
with $f(x) $ equal to $ x^\beta$ when $k=1$ and $(4/3)x^\beta$ when $k=2$.)

Theorem~\ref{main:iso} is proved in Section \ref{sec:iso}.
Section \ref{sec:main} derives the case $k=1$ of
Theorem \ref{thm:main} and then indicates the small changes needed for $k=2$,
and in passing derives Corollary \ref{cubesep} (see following Corollary~\ref{CorKPi}).
The easy proof of
Proposition~\ref{cube.prop} is given
in Section \ref{sec:cor}.

\section{Proof of Theorem \ref{main:iso}}\label{sec:iso}

\begin{lemma}\label{lem:plus1}
Let $X \subseteq V$ and let $f$ be a real-valued function on $V$. If
\begin{equation}\label{eq:1}
{1 \over \mu(X)} \int_X f^\beta d\mu = T^\beta,
\end{equation}
then
\begin{equation}\label{eq:2}
{1 \over \mu(X)} \int_X (f+1)^\beta d\mu \ge (T+1)^\beta.
\end{equation}
\end{lemma}

\begin{proof} Set $g(x)=f^\beta(x)$ for $x \in X$.  
Then the l.h.s. of (\ref{eq:1}) is $\mathbb Eg$ 
and the l.h.s. of (\ref{eq:2}) is $\mathbb E(g^{1/\beta}+1)^\beta$,
where $\E$ refers to uniform measure on $X$. 
But $p(x):=(x^{1/\beta}+1)^\beta$ is easily seen to be convex; so, by Jensen's inequality,
\[ \mathbb E(g^{1/\beta} +1)^\beta \ge ((\mathbb Eg)^{1/\beta}+1)^\beta, \]

\noindent which implies (\ref{eq:2}). \end{proof}

\mn 

The proof of Theorem \ref{main:iso} proceeds by induction on $n$. (This is also true of Theorem \ref{thm:T}, but beyond this the arguments seem to be different.) It is easy to see that the theorem holds for $n=1$, 
so we suppose $n \ge 2$.

Given $A$, fix an $i \in [n]$. Let
\[V_{0} = \{x \in V : x_i=0\},\]
\[V_{1} = \{x \in V : x_i=1\},\]
\[A_0=A \cap V_0,\]
and
\[A_1=(A \cap V_1)^i=\{x^i : x \in A, x_i=1\} \sub V_{0}.\]

Let $\mu'$ be uniform measure on $V_{0}$. For simplicity, write $h_0$ ($h_1$, $h$, resp.) for $h_{A_0}$ ($h_{A_1}$, $h_{A}$, resp.), a function on $V_{0}$ ($V_{0}$, $V$, resp.).

Let $\mu'(A_0)=a_0$, $\mu'(A_1)=a_1$, and $\mu(A)=a=(a_0+a_1)/2$. Then by induction hypothesis, for $i=0,1$,
\begin{equation}\label{ineq:indhyp}
\int h_i^\beta d\mu' \ge 2 a_i(1-a_i).
\end{equation}

\noindent 
We may assume $a_0 \ge a_1$. Note that
\begin{equation}\label{eq:cases}
h(x)=\begin{cases}h_0(x)+1& \mbox{ if } x \in A_0 \setminus A_1, \\ h_0(x) & \mbox{ if } x \in A_0 \cap A_1, \\ h_1(x^i)+1 &\mbox{ if } x^i \in A_1 \setminus A_0,\\h_1(x^i) &\mbox{ if } x^i \in A_0 \cap A_1;\end{cases}\end{equation}
so 
\beq{claim1}\begin{split}\int h^\beta d\mu &= \int_{A_0} h^\beta d\mu + \int_{(A_1)^i} h^{\beta} d\mu\\
&=\int_{A_0} h^\beta d\mu + \int_{A_1 \setminus A_0} (h_1+1)^\beta d\mu + \int_{A_0 \cap A_1} h_1^\beta d\mu\\
&\ge \int_{A_0} h^\beta d\mu + \int_{A_1} h_1^\beta d\mu\\
& \ge \int_{A_0} h^\beta d\mu +a_1(1-a_1) \end{split} \enq
(the last inequality by (\ref{ineq:indhyp})). Thus the theorem will follow if we show
\beq{stmt1} \int_{A_0} h^\beta d\mu \ge 2a(1-a)-a_1(1-a_1) ~(=a_0+a_1^2-(a_0+a_1)^2/2).\enq

The rest of this section is devoted to the proof of (\ref{stmt1}). Let $Z=\supp(h_0)\setminus A_1$ and $X=\supp(h_0) \cap A_1$ (see Figure \ref{fig1}); thus
\beq{obs} 2 \int_{A_0} h^\beta d\mu = \int_Z (h_0+1)^\beta d\mu' + \int_X h_0^\beta d\mu' + \int_{A_0 \setminus (A_1 \cup Z)} 1 d\mu'.\enq

\begin{figure}[h]
\begin{center}
\begin{tikzpicture}[scale=.8,every node/.style={scale=1}]

\filldraw[fill=black!10!white, draw=black] (1.5,0.1) ellipse (1.45cm and 1.45cm); 

\draw[draw=black] (0,0) ellipse (1.8cm and 1.8cm);

\draw[draw=black, thick] (0.1,0) ellipse (1cm and 1cm);

\node at (-1.65,-1.8) [label=above: {$A_0$}] {};

\node at (-.35,-.35) [label=above: {$Z$}] {};

\node at (.5,-.35) [label=above: {$X$}] {};

\node at (3,-1.3) [label=above: {$A_1$}] {};

\node at (-3,.5) [label=above: {$\supp(h_0)$}] {};

\node (A) at (-2.25,1.1) {};
\node (B) at (-.85,.3) {};

\draw [->] (A) -- (B);

\end{tikzpicture}
\end{center}
\caption{} \label{fig1}
\end{figure}

\begin{obs}\label{obs:a1} We may assume $A_1 \subseteq A_0$.\end{obs}

\begin{proof}
If there is $x \in A_1 \setminus A_0$ then we can find $y \in A_0 \setminus A_1$ since $\mu'(A_0) \ge \mu'(A_1)$. Let $B_1= (A_1 \setminus \{x\}) \cup \{y\}$, $B=A_0 \cup (B_1)^i$ and $B_0=B \cap V_0 ~ (=A_0)$. Notice that $|A|=|B|$, $|A_i|=|B_i|$ for $i \in \{0,1\}$, and
\[ \int_{B_0} h_B^\beta d\mu<\int_{A_0} h^\beta d\mu,\]
because: with $Z_B$ (resp. $X_B$) for $\supp(h_{B_0}) \setminus B_1$ (resp. $\supp(h_{B_0}) \cap B_1$), the location of $y$ changes either from $Z$ to $X_B$ or from $A_0 \setminus (A_1 \cup Z)$ to $(B_0 \cap B_1) \setminus X_B$. In either case its contribution to the r.h.s. of (\ref{obs}) shrinks. So if $A_1 \not\subseteq A_0$, then we can shift it to a ``worse'' set. \end{proof}

Let $\gs = \int_Z h_0^\beta d\mu'$, $\gamma = \int_X h_0^\beta d\mu'$, $\alpha=\gs + \gamma$
(=$\int h_0^\gb d\mu'$) and $\mu'(Z)=z$.
Since $A_1\sub A_0$, the r.h.s.\ of \eqref{obs} is
\begin{eqnarray} 
\int_Z(h_0+1)^\beta d\mu' + \gc +\mu'(A_0 \setminus (A_1 \cup Z)) &\ge& (\gs^{1/\beta}+z^{1/\beta})^\beta+\gc+(a_0-a_1-z)\nonumber\\
&=&
((\ga-\gc)^{1/\beta}+z^{1/\beta})^\beta+\gc+(a_0-a_1-z)\nonumber\\
&\geq& (\ga^{1/\beta}+z^{1/\beta})^\beta+(a_0-a_1-z),\label{alphabetaz}
\end{eqnarray}
where the first inequality is given by 
Lemma \ref{lem:plus1} and the second holds because
$((\ga-\gc)^{1/\beta}+z^{1/\beta})^\beta+\gc$ is
increasing in $\gamma$.

So we are done if we show that the expression in \eqref{alphabetaz} 
is at least
\beq{abd} 2(a_0+a_1^2)-(a_0+a_1)^2,\enq
where we are entitled to assume
\beq{alphaind}\alpha  = \int h_0^\beta d\mu' \ge 2a_0(1-a_0).\enq
(see \eqref{ineq:indhyp}) and
\beq{zmin}
z \le \min \{ \alpha, a_0-a_1\}
\enq
(where the second bound holds since $Z \subseteq A_0 \setminus A_1$). 
We consider two cases depending on which of $a_0-a_1$ and the r.h.s. of (\ref{alphaind}) is smaller.

\noindent\textbf{Case 1.} $2a_0(1-a_0) \le a_0-a_1$

Equivalently,
\begin{equation}\label{range2}a_1 \le a_0(2a_0-1).\end{equation}
Also, since $0 \le a_0(2a_0-1)$, we have 
\beq{range3}a_0 \ge 1/2.\enq

Note that \eqref{alphabetaz} is decreasing in $z$ and $z \le \alpha$ by (\ref{zmin}), so recalling that $2^\beta=3/2$ and using (\ref{alphaind}), we find that \eqref{alphabetaz} is at least
\beq{eq:case1} \alpha/2 +a_0-a_1 \ge a_0(1-a_0)+a_0-a_1.\enq

\noindent Subtracting (\ref{abd}) from (\ref{eq:case1}) gives
\[-a_1^2+(2a_0-1)a_1,\]
which is nonnegative since
\[f(x,y):=-y^2+(2x-1)y\ge 0 \quad  \mbox{for $x \in [\frac{1}{2},1]$ and $y \in [0,x(2x-1)]$.}\]
(Because: for any $y \ge 0$, $f(x,y)$ is nondecreasing in $x$, so it is enough to show the inequality holds when $y=x(2x-1)$, in which case $f(x,y)=x(1-x)(2x-1)^2 \ge 0$.)

\medskip
\noindent\textbf{Case 2.} $2a_0(1-a_0) \ge a_0-a_1$

Equivalently,
\beq{range1} a_0(2a_0-1) \le a_1 ~ (\le a_0).\enq

Again using the fact that \eqref{alphabetaz} is decreasing in $z$, now with $z \le a_0-a_1$ by (\ref{zmin}), we find that \eqref{alphabetaz} is at least
\beq{eq:13} (\alpha^{1/\beta}+(a_0-a_1)^{1/\beta})^\beta, \enq
which, in view of (\ref{alphaind}) (and the fact that (\ref{eq:13}) is increasing in $\alpha$), is at least
\beq{eq:case2} ((2a_0(1-a_0))^{1/\beta}+(a_0-a_1)^{1/\beta})^\beta.\enq
Thus the proof that \eqref{alphabetaz} is at least (\ref{abd}) in the present case is completed by the following proposition (applied with $x=a_0$ and $y=a_1$).
\begin{prop}\label{prop:gpos} Let
\[
g(x,y)=((2x(1-x))^{1/\beta}+(x-y)^{1/\beta})^\beta-2(x+y^2)+(x+y)^2.
\]
Then $g(x,y) \ge 0$ for $x,y \in [0,1] $ with $y\in [x(2x-1),x]$.

\end{prop}

\begin{proof} Observe that for $x \in [0,1]$,
\beq{eq:inicond2} g(x,x(2x-1))=x(1-x)(2x-1)^2 \ge 0, \enq
and
\beq{eq:inicond1} g(x,x)=0. \enq
Also, the partial derivative of $g(x,y)$ with respect to $y$ is
\[
g_y(x,y)=-(x-y)^{\frac{1}{\beta}-1}((2x(1-x))^{\frac{1}{\beta}}+(x-y)^{\frac{1}{\beta}})^{\beta-1}+2(x-y).
\]

Now, we claim that
\beq{cl:onezero} \mbox{ for given $x \in [0,1]$, $g_y(x,y)$ is equal to zero for at most one $y \in [x(2x-1),x)$.} \enq
Indeed, let $A=x-y \;(>0)$ and $B=2x(1-x)$. Then
\beq{num}
g_y(x,y)=0 \Leftrightarrow A^{\frac{1}{\beta}}+B^{\frac{1}{\beta}}=2^{\frac{1}{\beta-1}}A^{\frac{2\beta-1}{\beta(\beta-1)}}.
\enq
Notice that $A^{\frac{1}{\beta}}+B^{\frac{1}{\beta}}$ is increasing in $A$ while $2^{\frac{1}{\beta-1}}A^{\frac{2\beta-1}{\beta(\beta-1)}}$ is decreasing in $A$ (since $\frac{2\beta-1}{\beta(\beta-1)} < 0$). So we conclude that for any $B$, (\ref{num}) holds at most once, which is (\ref{cl:onezero}).

Finally, we claim that
\beq{cl:gpos}\mbox{for each $x \in (0,1)$, there is $c=c(x)>0$ such that $g(x,y)>0$ for all $y \in (x-c,x)$.}\enq
Note that Proposition \ref{prop:gpos} follows from the combination of (\ref{eq:inicond2}), (\ref{eq:inicond1}), (\ref{cl:onezero}), and (\ref{cl:gpos}).

\noindent\textit{Proof of (\ref{cl:gpos}).} Given $x \in (0,1)$, for $c=c(x)$ TBA,
\[ g(x,x-c)=((2x(1-x))^{\frac{1}{\beta}}+c^{\frac{1}{\beta}})^\beta +2x^2-2x-c^2,\]
so
\beq{ineq:1} g(x,x-c)>0 \Leftrightarrow ((2x(1-x))^{\frac{1}{\beta}}+c^{\frac{1}{\beta}})^\beta > c^2+2x(1-x).\enq
Now,
\[
((2x(1-x))^{\frac{1}{\beta}}+c^{\frac{1}{\beta}})^\beta = 2x(1-x)\left(1+\left(\frac{c}{2x(1-x)}\right)^{\frac{1}{\beta}}\right)^\beta,
\]
and if $c$ is small enough,
\[\begin{split}
\left(1+\left(\frac{c}{2x(1-x)}\right)^{\frac{1}{\beta}}\right)^\beta & = \exp[\Theta(c^{1/\beta})\beta]\\
&=1+\Theta(c^{1/\beta}),
\end{split}
\]
which implies (\ref{ineq:1}).
\end{proof}

\section{Proof of Theorem~\ref{thm:main}} \label{sec:main}

As noted at the end of Section~\ref{2ndapp}, we prove 
Theorem~\ref{thm:main} for $k=1$ and then indicate what changes for $k=2$.
This seemed to us slightly clearer than proving them together, though the differences are minor.
Extending to Theorem~\ref{thm:gen} is straightforward, though the counterpart of 
Proposition~\ref{prop:1toB} is slightly more painful than the original.

As usual, $A \sub V$ is \textit{increasing} if $x \in A$ and $y \ge x$ (with respect 
to the product order on $V$) imply $y \in A$ (and A is \textit{decreasing} is defined similarly). 
For $x,y$ with $x <y$, we write $x \lessdot y$ if $x \le z \le y$ implies $z \in \{x,y\}$.
We will need Harris' Inequality \cite{Harris}:
\begin{thm} \label{thm:H}
For any product measure $\nu$ on $Q_n$ and increasing $A, B \subseteq V$,
\[\nu(A \cap B) \ge \nu(A)\nu(B).\]
\end{thm}
Recall that $h_S$ was defined in (\ref{h}) and, for disjoint $A,B \sub V$, set
\[ h_{AB}(x)=\begin{cases} d_{B}(x) &\mbox{ if } x \in A, \\ 0 &\mbox{ if } x \notin A; \end{cases} \]
thus
\[\int_Ah_{V\sm B}d\mu = \int h_{AB}~ d\mu = {2^{-n}} |\nabla(A,B)|.\]

We need the following easy consequence of Theorem \ref{main:iso}.

\begin{cor} \label{CorKPi} If $(R,S,U)$ is a partition of $V$ with 
$\mu(R\cup U)=\alpha$, then 
\[ 
(2^{-n}|\nabla(R,S)|= \int_Rh_{R\cup U}d\mu \geq)~~\int_R h_{R\cup U}^\beta d\mu \ge 
2 \alpha(1-\alpha)- n^\beta \mu(U).
\]
\end{cor}

\begin{proof}
Theorem~\ref{main:iso} gives
\[
2\alpha(1-\alpha)\le \int h_{R\cup U}^\beta d\mu =
\int_R h_{R\cup U}^\beta d\mu + \int_U h_{R\cup U}^\beta d\mu \le \int_R h_{R\cup U}^\beta d\mu +n^\beta\mu(U),
\]
and the corollary follows.
\end{proof}

In particular, taking $(R,S,U) =(B,A,W)$ gives Corollary~\ref{cubesep}.\qed

We now assume the situation of Theorem \ref{thm:main}.
Note that each of $\mu(A)$, $ \mu(B) $ is $1/2\pm O(\eps)$.
In what follows we (abusively) use ``a.e.'' to mean ``all but an $O(\eps)$-fraction,'' so for example write ``a.e. $x \in A$ satisfies $Q$'' for ``$Q$ holds for all but an $O(\eps)$-fraction of the members of $A$.''

\begin{prop} \label{prop:1toB}
For a.e. $x \in A$, $h_{AB}(x)=1$.
\end{prop}

\begin{proof}
Applying Corollary~\ref{CorKPi} with $(R,S,U)=(A,B,W)$ (and using \eqref{nabla}) gives
\beq{1212}
(1+\eps)/2 \ge \int h_{AB} d\mu =\int_A h_{A\cup W}d\mu 
\ge \int_A h^\beta_{A\cup W} d\mu 
= 1/2 -O(\eps).
\enq
In particular, $\int (h_{AB}-h_{AB}^\beta)d\mu=O(\eps)$, which, since 
$\int(h_{AB}-h_{AB}^\beta)d\mu=\gO\left(\mu(\{x\in A: h_{AB}(x)\not\in\{0,1\})\right)$,
implies $h_{AB}(x) \in \{0,1\}$ for a.e.\ $x \in A$. \end{proof}

The next observation will allow us to assume that $A$ is increasing and $B$ is decreasing.

\begin{prop}\label{prop:shifting}
For any partition $(A, B, W)$ of $V$ there is another partition $(A', B', W')$ satisfying:

\begin{enumerate}
\item $\mu(X)=\mu(X')$ for $X \in \{A,B,W\}$;

\item $A'$ is increasing and $B'$ is decreasing;

\item $|\nabla_i(A,B)|\ge |\nabla_i(A',B')|$ for all $i \in [n]$.
\end{enumerate}
\end{prop}

\begin{proof}
This is a typical ``shifting" argument and we will be brief.
For $i\in [n]$, the $i$-\emph{shift} of a partition $(A,B,W)$ is defined thus:
let
\[V_{0} = \{x \in V : x_i=0\}, ~~  V_{1} = \{x \in V : x_i=1\},\] 
and for each $x \in V_0$ with $(x,x^i) \in (A,B), (A,W), \mbox{ or } (W,B)$, 
switch the affiliations of $x$ and $x^i$. 
This trivially does not change $|\nabla_i(A,B)|$, and it's easy to see that it 
does not increase $|\nabla_j(A,B)|$ for $j \in [n] \setminus \{i\}$.
(Consider the contribution to $\nabla_j(A,B)$ of any quadruple 
$\{x, x^i, x^j, (x^i)^j\}$.)

It is also clear that no sequence of nontrivial shifts 
can cycle (e.g.\ since any such shift strictly increases
$\sum_{x\in A}|x|-\sum_{x\in B}|x|$);
so there is a sequence that arrives at an $(A',B',W')$ 
stable under $i$-shifts (for all $i$), and this meets the requirements of the proposition.
\end{proof}

\begin{proof}[Proof of Theorem~\ref{thm:main}]

We first show there is an $i$ as in \eqref{conc}.
By Proposition \ref{prop:shifting}, we may assume $A$ is increasing and $B$ is decreasing. For each $i \in [n]$, let $A_i=\{x \in A: x^i \in B\}$, and notice that
\beq{lem:decr}\mbox{$A_i$ is a decreasing subset of $A$.}\enq
Indeed, given $x \in A_i$, consider any $y \in A$ satisfying $y \lessdot x$. Then $y^i \in B$ since $x^i \in B$ and $B$ is decreasing, so $y \in A_i$. 

By proposition \ref{prop:1toB}, 
\beq{oneA}\mbox{a.e. $x \in A$ is in exactly one $A_i$;}\enq
in particular, if we let $A_0=\{x \in A : d_B(x)=0\}$, then $\mu(A_0)=O(\eps)$.

Setting $\max \mu(A_i)=\mu(A)-\delta$, we just need to show that $\delta = O(\eps)$.

Assume (w.l.o.g.) that $\max \mu(A_i)=\mu(A_1)$, and let $\tilde A=\cup_{i \ne 1} A_i$, 
$C_1=A \setminus A_1$, and $\tilde C=A \setminus \tilde A$. By (\ref{oneA}),
\beq{tC}
\mu(\tilde C)\ge \mu(A_1)-O(\eps),
\enq
while $C_1 \cap \tilde C = A_0$ implies
\[   
\mu(C_1 \cap \tilde C)=O(\eps).
\]    
Moreover, \eqref{lem:decr} and the fact that $A$ is increasing imply that 
$C_1$ and $\tilde C$ are increasing (in $V$); so Theorem \ref{thm:H} gives
\beq{eq:Harris}
O(\eps)= \mu(C_1 \cap \tilde C) \ge \mu(C_1)\mu(\tilde C) \ge \delta(\mu(A)-\delta-O(\eps)),
\enq
whence
\[ 
\delta = O(\eps) ~\mbox{ or }~ \mu(A)-\delta-O(\eps)=O(\eps).
\]
But $\delta=O(\eps)$ is what we want, so we may assume for a contradiction that $\mu(A)-\delta-O(\eps)=O(\eps)$; equivalently, $\mu(A_1) = O(\eps)$. In this case, 
$\mu(A_i) = O(\eps)$ for all $i $, so there is a partition $[n]=I \cup J$ such that each of 
$A_I$ ($:=\cup_{i \in I} A_i$) and $A_J$ has measure $\mu(A)/2 + O(\eps)$. 
But then, setting $C_I=A \setminus A_I$ and $C_J=A \setminus A_J$, and again using  
Theorem~\ref{thm:H}, we have
\[ O(\eps) = \mu(C_I \cap C_J) \ge \mu(C_I)\mu(C_J) \ge \mu^2(A)/4-O(\eps),\]
which is impossible. \end{proof}

For \eqref{conc'}, let $i$ be as above and for $\pi \in \{0,1\}$, let $C(i,\pi)=\{v:v_i=\pi\}$. 
If $D$ is one of these subcubes then with $|A \cap D|=\delta 2^{n-1}$, 
Corollary~\ref{CorKPi} (applied in $D$ with $R=A \cap D$ and $U=W \cap D$) gives at least $[2\delta(1-\delta)-O(\eps)]2^{n-1}$ edges in $\nabla(A,B) \setminus \nabla_iA$, which with \eqref{nabla} and \eqref{conc} forces $\delta$ to be either $O(\eps)$ or $1-O(\eps)$.
So exactly one, say $C$, has $\delta = 1-O(\eps)$, and this $C$ satisfies 
\eqref{conc'}.\qed

\nin
\emph{Changes for $k=2$ (briefly).}
The only changes are to Proposition~\ref{prop:1toB} and the final argument(s).
For the former, the statement is now:
\[
\mbox{for a.e. $x \in A$, $~h_{AB}(x)=2$.}
\]
Set $f(x)= (4/3)x^\beta$.
Theorem~\ref{main:iso} gives 
$\int f(h_{A\cup W})d\mu \geq 1/2-O(\eps)$, leading to 
\[
\int f(h_{AB})d\mu \geq 1/2 -O(\eps).
\]

Now let $X(x) =h_{AB}(x)$ for $x\in A$ and write $\E$ for expectation w.r.t.\ 
uniform measure on $A$.  Our assumptions on $\mu(A)$ and $|\nabla(A,B)|$ give
\[
\E X = \frac{1}{\mu(A)}\int h_{AB}d\mu = \frac{|\nabla(A,B)|}{\mu(A)2^n} \leq 2+O(\eps),
\]
so, using the concavity of $f$, we have
\[
\int f(h_{AB})d\mu = \mu(A)\E f(X) \leq \mu(A)f(\E X) \leq 1/2+O(\eps).
\]
It's then easy to see (if somewhat annoying to write) that concavity of $f$, with 
$\E f(X)-f(\E X) =O(\eps)$ and $f(\E X) = 2\pm O(\eps)$ (and $X\in \mathbb Z$) implies,
first, that there is a $c$ such that $f(x)=c$ for a.e.\ $x\in A$, and, second, that $c=2$.

For the step leading to \eqref{conc}
we may as well think of a general $k$.  Thus we assume
$A$ and $B$ are increasing and decreasing (resp.), with $n^\beta\mu(W)\leq \eps$,
$\mu(A) = (1\pm \eps)2^{-k}$, $|\nabla(A,B)| < (1+\eps)k2^{n-k}$,
and $h_{AB}(x)=k$ for a.e.\ $x\in A$, 
and want to show
\[
\mbox{there is $I\sub [n]$ of size $k$ such that
$|\nabla_i A| \ge (1-O(\eps)) 2^{n-k}~\forall i\in I.$}
\]
Here for each $k$-subset $I$ of $[n]$ we set
\[
A_I=\{ x \in A : x^{i}\in B~\forall i\in I\}.
\] 
Each $A_I$ is decreasing in $A$ and a.e.\ $x\in A$ is in exactly one $A_I$.
We then assume $\max_I\mu(A_I) = \mu(A_{[k]}) = \mu(A)-\delta$
and continue essentially as before.

The step yielding \eqref{conc'} again takes no extra effort for general $k$:
here we have $2^k$ subcubes corresponding to the members of $\{0,1\}^k$,
and Corollary~\ref{CorKPi} (with \eqref{nabla} and \eqref{conc}) shows that all but
one of these meet $A$ in sets of size $O(\eps)2^{n-k}$ (and the one that doesn't is 
the promised $C$).

\section{Proof of Proposition~\ref{cube.prop}}\label{sec:cor}

Let $|A|=a$.
For $z\in \{0,1\}^I$ let $V_z=\{x:x_i=z_i~\forall i\in I\}$,
$A_z=A\cap V_z$, $a_z=|A_z|$ and $\ga_z=a_z/a$.
Assume (w.l.o.g.) that $a_z$ is maximum when $z=\underline 0$.
We have

\[\begin{split}\eps a & \ge |\nabla A\sm \nabla_I A| 
= \sum_z|\nabla (A_z, V_z \setminus A_z)|
\ge \sum_z a_z\log_2(2^{n-k}/a_z)\\
&  
=a\left[H(\ga_z:z\in \{0,1\}^I)+     \log_2(2^{n-k}/a)\right]
=a H(\ga_z:z\in \{0,1\}^I)+     O(\eps)a, 
\end{split}\]
where $H$ is binary entropy and 
the inequality is given by (\ref{edge.iso}). 
It follows that each $\ga_z$ is either $O(\eps/\log(1/\eps))$ or $1-O(\eps)$ ; so in fact 
$\ga_{\underline{0}}= 1-O(\eps/\log (1/\eps))$ 
and $V_{\underline{0}}$ is the promised subcube.


\begin{thebibliography}{AAAAA}

\bibitem{Ellis}
D. Ellis,
Almost Isoperimetric Subsets of the Discrete Cube,
pp. 363-380 in \emph{Combin. Probab. and Comput.} \textbf{20}
2011.

\bibitem{EKL}
D.\ Ellis, N.\ Keller and N.\ Lifshitz,
On the structure of subsets of the discrete cube with small edge boundary,
\emph{Discrete Analysis} 2018:9.

\bibitem{FKN}
E. Friedgut, G. Kalai and A. Naor,
Boolean functions whose Fourier transform is concentrated on the first two levels,
pp. 427-437 in \emph{Adv. Appl. Math.} \textbf{29}
2002.

\bibitem{Harper}
L. H. Harper,
Optimal numberings and isoperimetric problems on graphs,
pp. 385-393 in \emph{J. Combinatorial Theory} \textbf{1}
1966.

\bibitem{Harris}
T. E. Harris,
A lower bound for the critical probability in a certain percolation process,
pp. 13-20 in \emph{Proc. Cambridge. Philos. Soc.} \textbf{56}
1960

\bibitem{IK} L. Ilinca and J. Kahn,
Counting maximal antichains and independent sets,
{\em Order} {\bf 30} (2013), 427-435.

\bibitem{KKL}
J. Kahn, G. Kalai and N. Linial,
The influence of variables on Boolean functions: Extended
abstract, {\em Proc. 29th IEEE Symposium
on Foundations of Computer Science\/}, 1988.



\bibitem{misQ}
J. Kahn and J. Park, The number of maximal independent sets in the Hamming cube, preprint.
arXiv:1909.04283 [math.CO]

\bibitem{Lindsey}
J. H.\ Lindsey II, Assignment of numbers to vertices, \emph{Amer.\ Math.\ Monthly} 
\textbf{71} (1964), 508-516.

\bibitem{Margulis}
G.A.\ Margulis, Probabilistic characteristics of graphs with large connectivity,
\emph{Probl.\ Peredachi Inf.} \textbf{10} (1974), 101-108.
English version:
\emph{Problems Info.\ Transmission} \textbf{10}
(1977), 174-179.



\bibitem{TalRusso}
M. Talagrand,
On Russo's approximate 0-1 law,
pp. 1576-1587 in \emph{Ann. Probab.} \textbf{22}
1994

\bibitem{Tal93}
M. Talagrand,
Isoperimetry, Logarithmic Sobolev Inequalities on the Discrete Cube, and Margulis' Graph Connectivity Theorem,
pp. 295-314 in \emph{Geometric and Functional Analysis} \textbf{3},
1993


\end{thebibliography}
\end{document}